\documentclass[leqno]{amsart}
\usepackage{amsfonts,amssymb,amsmath,amsgen,amsthm}
\usepackage{hyperref}

\theoremstyle{plain}
\newtheorem{theorem}{Theorem}[section]
\newtheorem{definition}[theorem]{Definition}
\newtheorem{assumption}[theorem]{Assumption}
\newtheorem{lemma}[theorem]{Lemma}
\newtheorem{corollary}[theorem]{Corollary}
\newtheorem{proposition}[theorem]{Proposition}

\theoremstyle{remark}
\newtheorem{remark}[theorem]{Remark}
\newtheorem*{notation}{Notation}
\newtheorem{example}[theorem]{Example}

\def\Tend#1#2{\mathop{\longrightarrow}\limits_{#1\rightarrow#2}}

\font\tenms=msbm10
\font\sevenms=msbm7
\font\fivems=msbm5
\newfam\bbfam \textfont\bbfam=\tenms \scriptfont\bbfam=\sevenms
\scriptscriptfont\bbfam=\fivems

\newcommand{\beq}{\begin{eqnarray}}
\newcommand{\eeq}{\end{eqnarray}}
\newcommand{\bq}{\begin{equation}}
\newcommand{\eq}{\end{equation}}
\newcommand{\beqn}{\begin{eqnarray*}}
\newcommand{\eeqn}{\end{eqnarray*}}

\def\DD{\mathop{\bf D\kern 0pt}\nolimits}

\def\SS{\mathop{\bf S\kern 0pt}\nolimits}
\def\ZZ{\mathop{\bf Z\kern 0pt}\nolimits}
\def\TT{\mathop{\bf T\kern 0pt}\nolimits}
\def\virgp{\raise 2pt\hbox{,}}
\def\cdotpv{\raise 1pt\hbox{ ;}}

\def\eps{\varepsilon}
\def\beq{\begin{equation}}
\def\eeq{\end{equation}}

\def\cdotv{\raise 2pt\hbox{,}}

\def\C{{\mathbf C}}
\def\R{{\mathbf R}}
\def\N{{\mathbf N}}
\def\Tc{{\mathcal T}}
\def\Sch{{\mathcal S}}
\def\F{\mathcal F}
\def\O{\mathcal O}

\def\virgp{\raise 2pt\hbox{,}}

\def\({\left(}
\def\){\right)}
\def\<{\left\langle}
\def\>{\right\rangle}
\def\le{\leqslant}
\def\ge{\geqslant}

\def\Tend#1#2{\mathop{\longrightarrow}\limits_{#1\rightarrow#2}}

\def\d{{\partial}}
\def\l{\lambda}

\numberwithin{equation}{section}

\begin{document}

\title[A Nonlinear adiabatic Theorem for Coherent States]{A Nonlinear
  Adiabatic Theorem for Coherent States} 
\author[R. Carles]{R\'emi Carles}
\address[R. Carles]{Univ. Montpellier~2\\Math\'ematiques
\\CC~051\\F-34095 Montpellier}
\address{ CNRS, UMR 5149\\  F-34095 Montpellier\\ France}
\email{Remi.Carles@math.cnrs.fr}
\author[C. Fermanian]{Clotilde~Fermanian-Kammerer}
\address[C. Fermanian]{LAMA UMR CNRS 8050,
Universit\'e Paris EST\\
61, avenue du G\'en\'eral de Gaulle\\
94010 Cr\'eteil Cedex\\ France}
\email{Clotilde.Fermanian@univ-paris12.fr}
\begin{abstract}
  We consider the propagation of wave packets for a one-dimensional
  nonlinear Schr\"odinger equation with a matrix-valued potential, in
  the semi-classical  limit. For an initial coherent state polarized
  along some eigenvector, we prove that the nonlinear evolution
  preserves the separation of modes, in a scaling such that
  nonlinear effects are critical (the envelope equation is
  nonlinear). The proof relies on a fine geometric analysis of the
  role of spectral projectors, which is compatible with
  the treatment of nonlinearities.  We also prove a nonlinear superposition principle for these adiabatic wave packets.
\end{abstract}
\thanks{This work was supported by the French ANR project
  R.A.S. (ANR-08-JCJC-0124-01)}  
\maketitle

\section{Introduction}
\label{sec:intro}
We consider the semi-classical limit $\eps\to 0$ for the nonlinear
Schr\"odinger equation 
\begin{equation}\label{eq:NLS0}
 \left\{
\begin{aligned}
     i\eps\d_t \psi^\eps +\frac{\eps^2}{2}\d_x^2 \psi^\eps&=V(x)\psi^\eps
     + \Lambda |\psi^\eps |_{\C^N}^{2} \psi^\eps, \quad
(t,x)\in
    {\R}\times {\R}, \\
     \psi^\eps_{\mid t=0} &= \psi^\eps_{0}
  \end{aligned}
\right.
\end{equation}
where  $\Lambda\in\R$.   The data $\psi^\eps_0$ and the solution
$\psi^\eps(t)$ are vectors of $\C^N$, $N\ge  1$.  
The quantity $|\psi^\eps|_{\C^N}^2$ denotes the square of the
Hermitian norm in $\C^N$ of the vector~$\psi^\eps$. Finally, the 
potential $V$ is smooth and valued in the set of $N$ by $N$ Hermitian matrices.
Such systems appear in the modelling of Bose-Einstein condensate
(see~\cite{Bao} and references therein).

\smallbreak

\begin{definition}\label{def:sousquad}
We say that a function $f$ is \emph{at most quadratic} if  $f\in{\mathcal
  C}^\infty(\R)$ and for all  
$ k\ge 2,\;\;f^{(k)}\in L^\infty(\R).$
\end{definition}

We make the following assumptions on the potential $V$:

 \begin{assumption}\label{assumption}
$(1)$ We have $V(x)=D(x)+W(x)$ with $D,W\in{\mathcal
  C}^\infty(\R,\R^{N\times N})$, $D$ diagonal with at most quadratic
coefficients, and $W$  symmetric  and bounded as well as its
derivatives, $W\in W^{\infty,\infty}(\R)$.\\ 
$(2)$ The matrix $V$ has $P$ distinct, at most quadratic,  eigenvalues
$\lambda_1,\dots,\lambda_P$  and  
\begin{equation}\label{gapcondition}
\exists c_0, n_0\in\R^+,\;\;\forall j\not=k,\;\;\forall x\in\R,\;\;
\left|\lambda_j(x)-\lambda_k(x)\right|\ge\,c_0 \< x\>^{-n_0}. 
\end{equation}
 \end{assumption} 

$ $

Under these assumptions (the first point suffices), we can prove
global existence of the solution $\psi^\eps$ for fixed $\eps>0$: 

 \begin{lemma}\label{lem:existpsi}
 If $V$ satisfies Assumption~\ref{assumption} and $\psi^\eps_0\in
 L^2(\R)$, there exists a unique, global, solution to~\eqref{eq:NLS0} 
$$\psi^\eps\in {\mathcal C}\left(\R; L^2(\R)\right)\cap L^8_{\rm
  loc}\(\R;L^4(\R)\).$$ 
The $L^2$-norm of $\psi^\eps$ does not depend on time:
$\|\psi^\eps(t)\|_{L^2(\R)}=\|\psi^\eps_0\|_{L^2(\R)},\;\;\forall
t\in\R.$ 
 \end{lemma}

The proof of this lemma is sketched in Appendix~\ref{sec:appA}.

In this nonlinear setting, the size of the initial data is crucial. 
As in \cite{CF-p}, we choose to consider initial data of order $1$ (in
$L^2$), and to introduce a dependence upon $\eps$ in the coupling
constant (note that the nonlinearity is homogeneous). This leads to
the equation   
$$  i\eps\d_t \psi^\eps +\frac{\eps^2}{2}\d_x^2 \psi^\eps
     =V(x)\psi^\eps+ \Lambda\eps^{2\beta} |\psi^\eps |_{\C^N}^{2} \psi^\eps, $$
     and we choose the exponent $\beta=3/4$, which is critical
     for the type of initial data we want to consider (coherent state)
     when the  potential $V$
     is \emph{scalar} (see \cite{CF-p}). 
  We are left with the nonlinear semi-classical Schr\"odinger equation 
     \begin{equation}\label{eq:NLS}
     i\eps\d_t \psi^\eps +\frac{\eps^2}{2}\d_x^2 \psi^\eps
     =V(x)\psi^\eps+ \Lambda\eps^{3/2} |\psi^\eps |_{\C^N}^{2}
     \psi^\eps \quad ; \quad 
\psi^\eps_{\mid t=0} = \psi^\eps_{0}.
\end{equation}
We focus on initial data which are perturbation of wave packets 
\begin{equation}\label{data}
\psi^\eps_0(x)= \eps^{-1/4} e^{i\xi_0
    (x-x_0)/\eps}a
\left(\frac{x-x_0}{\sqrt\eps}\right) \chi(x)+r_0^\eps(x),
\end{equation}
where the initial error satisfies
\begin{equation}
  \label{eq:erreur-init}
  \|r_0^\eps\|_{L^2(\R)} +  \|x r_0^\eps\|_{L^2(\R)} +
  \|\eps \d_x r_0^\eps\|_{L^2(\R)}=  \O(\eps^\kappa)
  \text{ for some }\kappa>\frac{1}{4}.
\end{equation}
The
profile $a$ belongs to the Schwartz class, $a\in{\mathcal S}(\R;\C)$,
and the initial datum is polarized along an eigenvector 
$\chi(x)\in{\mathcal C} ^\infty(\R;\C^N)$:
$$V(x)\chi(x)=\lambda_1(x) \chi(x),\quad \text{ with }| \chi(x)|_{\C^N}=1.$$ 
Note that $\l_1$ is simply a notation for \emph{some} eigenvalue, up
to a renumbering of eigenvalues. The $L^2$-norm of
$\psi_0^\eps$ is independent of $\eps$, $\|\psi_0^\eps\|_{L^2(\R)}=
\|a\|_{L^2(\R)}$. As pointed out above, this is equivalent to
considering \eqref{eq:NLS0} with initial data of the same form
\eqref{data}, but of order $\eps^{3/4}$ in $L^2(\R)$. 
The evolution of such data when $a$ is a Gaussian has been
extensively studied by G.~Hagedorn on the one hand,  and  by
G.~Hagedorn and A.~Joye on the other hand, in the linear context
$\Lambda=0$ (see~\cite{Hag94,HJ98}). 
These data are also particularly interesting for numerics (see
\cite{Lubich} and the references therein). 
\smallbreak

Because of the gap condition,  the matrix $V$ has smooth
eigenvalues and eigenprojectors 
(see~\cite{Kato}). Besides, the gap condition~\eqref{gapcondition}
also implies that we control the growth of the eigenprojectors (see
Lemma~\ref{lem:projectors}). Note however than in dimension~$1$ ($x\in\R$), one can have smooth eigenprojectors without any gap condition. We explain this fact below and  give an example of
projectors that we can consider; we also illustrate why things may be more
complicated in higher dimensions ($d\ge 2$). 

\begin{example}
For $N=2$ and $x\in \R$, consider
\begin{equation}\label{example}
V(x)= (ax^2 +b){\rm Id} +\left(\begin{array}{cc} u(x) & v(x) \\ v(x) &
    -u(x)\end{array}\right), 
\end{equation}
for $a,b\in\R$, and $u$ and $v$ smooth and bounded with bounded
derivatives. Such a potential satisfies Assumption~\ref{assumption}. Its
eigenvalues are the two functions  
$$\lambda^\pm(x)= ax^2+b \pm\sqrt{u(x)^2+v(x)^2}.$$
These functions are clearly smooth outside the set of points $x_0$
such that $u(x_0)^2+v(x_0)^2=0$. Besides, for such points, one can
renumber the modes in order to build smooth eigenvalues. More
precisely, observe first that if $u(x)^2+v(x)^2 =O((x-x_0)^\infty)$
close to $x_0$, the functions $\lambda^\pm$ are smooth close to
$x_0$. Moreover,  if  $u(x)^2+v(x)^2 =(x-x_0)^k f(x)$ with
$f(x_0)\not=0$, necessarily $f(x_0)>0$ and $k=2p$, so we have  
$$\lambda^\pm(x)=ax^2+b \pm|x-x_0|^p\sqrt{f(x)}.$$
For $p$ even these functions are again smooth. However, when $p$ is
odd, they are no longer smooth and we perform a renumbering of the
eigenfunctions, observing that  
$$x\mapsto ax^2+b +(x-x_0)^p\sqrt{f(x)}$$ 
are smooth eigenvalues of $V$ close to $x_0$.
\end{example}
\begin{example}
Resume the above example, with now $x\in \R^d$,
  $d\ge 2$. The smoothness of the eigenvalues  is no longer guaranteed:
suppose $u(x)=x_1$ and $v(x)=x_2$, then the functions $\lambda_\pm$
are not smooth and one cannot find any renumbering  which makes them
smooth. 
\end{example}

\begin{example}
For an example of a potential which
satisfies~\ref{assumption}, we simply choose $V$ as in~\eqref{example} with  
$$ c_u u(x)=c_v v(x)=\< x\>^{-n_0}, \quad c_u^2+c_v^2\not =0.$$ 
\end{example}

\subsection{The ansatz}

We consider
the classical trajectories $\left(x(t),\xi(t)\right)$ solutions to
\begin{equation}\label{eq:traj}
 \dot x(t)=\xi(t),\;\;\dot \xi(t)=-\nabla
 \lambda_1(x(t)),\;\;x(0)=x_0,\;\xi(0)=\xi_0. 
 \end{equation}
Because  $\lambda_1$ is at most quadratic,  the classical
trajectories grow at most exponentially in time (see e.g. \cite{CF-p}):
\begin{equation}\label{expgrowth}
\exists C>0,\;\; |\xi(t)|+|x(t)|\lesssim {\rm e} ^{Ct}.
\end{equation}  
 We denote by~$S$
the action associated with $(x(t),\xi(t))$
\begin{equation}\label{eq:action}
S(t)=\int_0^t \left( \frac{1}{2} |\xi(s)|^2-\lambda_1(x(s))\right)\,ds.
\end{equation}
 We consider the function  $u=u(t,y)$  solution to  
 \begin{equation}\label{eq:profil}
 i\partial_t  u +\frac{1}{2} \d_y^2 u=\frac{1}{2}\lambda''_1\(x(t)\) y^2
u+\Lambda |u|^2u\quad ;\quad 
u(0,y)=a(y),
\end{equation}
 and we denote by 
$\varphi^\eps$  the function associated with~$u,
x,\xi, S$ by:
\begin{equation}\label{eq:phi}
\varphi^\eps(t,x)=\eps^{-1/4} u
\left(t,\frac{x-x(t)}{\sqrt\eps}\right)e^{i\left(S(t)+\xi(t)
    (x-x(t))\right)/\eps}.
\end{equation}
Global existence of $u$ and control of its derivatives and momenta are
proved in~\cite{Ca-p}. More precisely, we have the following result.

\begin{theorem}[From \cite{Ca-p}]\label{theo:profil}
Suppose $a\in{\mathcal S}(\R)$. There exists a  unique, global
solution $u \in {\mathcal C}(\R;L^2(\R))\cap L^8_{\rm
  loc}(\R;L^4(\R))$ to~\eqref{eq:profil}. In addition, for all $k,p\in
\N$, $\<y\>^k\partial_y ^p u \in {\mathcal C}(\R;L^2(\R))$ and 
\begin{equation}\label{boundedmoments}
\forall k,p\in\N,\;\;\exists C>0,\;\;\forall
t\in\R^+,\;\;\|\<y\>^k\partial_y ^p u(t,\cdot)\|_{L^2(\R)}\lesssim
{\rm e}^{Ct}. 
\end{equation}
\end{theorem}
In  particular, note that $\partial_y^p u(t,\cdot)$ is in $L^\infty$
for all $p\in\N$.  
These results have consequences on $\varphi^\eps$. As far as
the $L^\infty$ norm is concerned, we infer, using \eqref{expgrowth},  
\begin{equation}\label{Linftybound}
\forall p\in\N ,\;\;
\| (\eps\partial_x)^p\varphi^\eps(t)\|_{L^\infty}
\lesssim \eps^{-1/4}{\rm e}^{C_p t}. 
\end{equation}

 We use the  time-dependent eigenvectors constructed in \cite{Hag94}
 (see also \cite{HJ98} and~\cite{Teufel}). To make  the notations precise,
 we denote by $d_j$ the multiplicity of the eigenvalue~$\lambda_j$,
 $1\le j\le P$ 
 (note that $\sum_{1\le j\le P} d_j=N$).

 \begin{proposition} \label{prop:timedepeigen}
 There exists a smooth orthonormal family
 $\left(\chi^{\ell}(t,x)\right)_{1\le\ell\le d_1}$  such that  for
 all $t$,
 $\left(\chi^{\ell}(t,x)\right)_{1\le\ell\le d_1}$ spans the
 eigenspace associated to $\lambda_1$,  
 $\chi^1(0,x)=\chi(x)$ and for $m\in\{1,\cdots, d_1\}$,
 \begin{equation}\label{eq:eigenvector}
\left(\chi^m(t,x),\partial_t \chi^\ell(t,x) + \xi(t)\partial_x \chi
  ^\ell(t,x)\right)_{\C^N}=0. 
\end{equation}
Moreover, 
for $\ell\in\{1,\dots,d_1\}$, $k,p\in\N$, there exists a constant
$C=C(p,k)$ such that  
$$
\left\lvert \partial_t^p \partial_x^k\chi ^\ell
  (t,x)\right\rvert_{\C^N}\le C \, {\rm 
  e}^{C t} \<x\>^{(k+p)(1+n_0)}, $$
where $n_0$ appears in \eqref{gapcondition}. 
\end{proposition}

Note that equation~(\ref{eq:eigenvector}) for $m=\ell$ is true as soon
as the eigenvector $\chi^\ell$ is normalized and real-valued. 

Equation~\eqref{eq:eigenvector} is often referred to as 
\emph{parallel transport}.   These time-dependent eigenvectors are
commonly used in adiabatic theory  and are connected with the Berry
phase (see~\cite{Teufel}).  
Their construction is recalled in Section~\ref{sec:eigenvector}, where
the control of their growth is also established.

  \begin{notation}
In the case of a single coherent state, we complete the family
$\left(\chi^{\ell}(t,x)\right)_{1\le\ell\le 
  d_1}$ as an orthonormal basis $\(\chi^\ell_j\)_{{1\le j\le
    P}\atop{1\le \ell\le d_j}}$ of $\C^N$ as follows: 
\begin{itemize}
\item $\chi_1^\ell=\chi^\ell$,
\item For $j\ge 2$ and $1\le \ell\le d_j$,
  $\chi^\ell_j=\chi^\ell_j(x)$ does not depend on time, 
\item  For  $j\ge 2$,  $(\chi^\ell_j)_{1\le j\le d_j}$ spans the
  eigenspace associated to $\lambda_j$. 
\end{itemize}
\end{notation} 
 
 \subsection{The results}
 
 We  prove that there is adiabatic decoupling for the solution
 of~(\ref{eq:NLS}) with initial data which are coherent states of the
 form~(\ref{data}): the solution keeps the same form and remains in
 the same eigenspace.  
 
 \begin{theorem} \label{theo:matrix} Let $a\in \Sch(\R)$ and $r_0^\eps$
  satisfying \eqref{eq:erreur-init}. Under
  Assumption~\ref{assumption}, consider $\psi^\eps$ solution to
  the Cauchy problem \eqref{eq:NLS}--\eqref{data}, and the approximate
  solution $\varphi^\eps$ given by \eqref{eq:phi}. There exists a constant
  $C>0$ such that the function
$$w^\eps(t,x)=\psi^\eps(t,x)-\varphi^\eps(t,x) \chi^1(t,x),$$
where $\chi^1$ is given by Proposition~\ref{prop:timedepeigen}, satisfies 
$$\sup_{|t|\le C{\log}{\log}\frac{1}{\eps}}
\(\|w^\eps(t)\|_{L^2}+\|x w^\eps(t)\|_{L^2}+ \|\eps\d_x
w^\eps(t)\|_{L^2}\)\Tend\eps 0 0 .$$ 
\end{theorem}

This adiabatic decoupling between the modes is well-known in the
linear setting and is at the basis of numerous results on
semi-classical  Schr\"odinger operator with matrix-valued potential 
 in the framework of Born-Oppenheimer
approximation for molecular dynamics. On this subject, the reader can consult the article of
H.~Spohn and S.~Teufel~\cite{SpohnTeufel} or the book of S.~Teufel
\cite{Teufel} for a review on the topic (see also \cite{BenAbdMehats}
for an adiabatic result in a nonlinear context  and  \cite{Jecko1} for
application of adiabatic theory to the obtention  of resolvent
estimates).

\begin{remark}\label{rem:log}
  Suppose that $V$
  depends on $\eps$ with $V^\eps = D+\eps W$, where $D$ and $W$ are as
  in Assumption~\ref{assumption}; this is so in the model presented in
  \cite{Bao}. Then the above result remains true for $|t|\le C\log
  \left(\frac{1}{\eps}\right)$: we gain one logarithm. See
  Remark~\ref{rem:whylog} for the key arguments. Also, the assumption
  on the initial error can be relaxed: to prove the analogue of
  Theorem~\ref{theo:matrix} with an approximation in $L^2$ up to
  $C\log 1/\eps$, \eqref{eq:erreur-init} can be replaced with 
  \begin{equation*}
    \|r_0^\eps\|_{L^2(\R)}\to 0\text{ as }\eps \to 0.
  \end{equation*}
In contrast with the general framework of this paper, no
rate is needed: the rate in \eqref{eq:erreur-init} is due to the fact
that we cannot use Strichartz estimates here. 
\end{remark}

It is also interesting to analyze the evolution of solution associated
with data which are the superposition of two data of the studied
form. We suppose  
$$\psi^\eps_0(x)=\varphi^\eps_{1}(0,x) \chi_1(x)+
\varphi^\eps_{2}(0,x)\chi_2(x),$$ 
where both functions $\varphi^\eps_{1}$ and $\varphi^\eps_{2}$ have
the form \eqref{eq:phi}, for two eigenvectors of $V$, $\chi_1$ and
$\chi_2$, and phase space points $\left(x_1,\xi_1\right)$ and
$\left(x_2,\xi_2\right)$. We assume
\begin{equation*}
  \(\chi_1,x_1,\xi_1\)\neq \(\chi_2,x_2,\xi_2\). 
\end{equation*}
We associate with the phase space points $(x_j,\xi_j)$, $j\in\{1,2\}$
the classical trajectories $(x_j(t),\xi_j(t))$, and the action $S_j(t)$
associated with $\tilde\lambda_j$ such that
$$V(x)\chi_j(x)=\tilde\lambda_j(x)\chi_j(x).$$
Note that we may have
$\tilde\lambda_1=\tilde\lambda_2$. Let us denote by
$\chi_j^\ell(t)_{\scriptstyle{1\le j\le P}\atop\scriptstyle{1\le
    \ell\le d_j}}$ a time-dependent orthonormal basis of eigenvectors
defined according to Proposition~\ref{prop:timedepeigenbis}  (see also
Proposition~\ref{prop:timedepeigen} above) with
$\chi_1^1(0,x)=\chi_1(x)$, $\chi_2(x)= \chi_1^2(0,x)$ if
$\tilde\lambda_1=\tilde\lambda_2$, $\chi_2(x)= \chi_2^1(0,x)$
otherwise, and by $\varphi^\eps_j$ the ansatz defined 
by~\eqref{eq:phi}. To unify the presentation, we write
\begin{equation*}
  \chi^1= \chi_1^1\quad ;\quad
\chi^2 = 
\left\{
  \begin{aligned}
    \chi_1^2 &\text{ if }\tilde\lambda_1=\tilde\lambda_2,\\
\chi_2^1& \text{ otherwise.}
  \end{aligned}
\right.
\end{equation*}

\begin{theorem}\label{theo:superposition1}
Set $E_j= {\xi_j^2\over 2}+\tilde\lambda_j(x_j)$ for $j\in\{1,2\}$ and 
suppose
\begin{equation*}
  \Gamma = \inf_{x\in \R}   \left\lvert
  \tilde\lambda_1(x)-\tilde\lambda_2(x)-(E_1-E_2)\right\rvert >0.
\end{equation*}
There exists  $C>0$ such that the function
$$w^\eps(t)=\psi^\eps(t)-\varphi^\eps_{1}\chi^1(t,x)-
\varphi^\eps_{2}\chi^2(t,x).$$ 
satisfies
$$\sup_{t\le C{\log}{\log}{1\over\eps}}
\(\|w^\eps(t)\|_{L^2}+\|x w^\eps(t)\|_{L^2}+ \|\eps\d_x
w^\eps(t)\|_{L^2}\)\Tend\eps 0 0 .$$
\end{theorem} 

Note that if $\tilde \lambda_1=\tilde\lambda_2$, one recovers the
condition $E_1\not=E_2$ of~\cite{CF-p}.  
The proof of Theorem~\ref{theo:superposition1} follows the same lines
as in  \cite[Section~6]{CF-p}. The constant $\Gamma$ controls the
frequencies of time interval where trajectories cross. 

\begin{remark}
In finite time, the situation is different whether
$\tilde\lambda_1=\tilde\lambda_2$ or not. If
$\tilde\lambda_1=\tilde\lambda_2$, the superposition holds in finite
time without any condition on $\Gamma$; this comes from the fact that
the trajectories $x_1(t)$ and $x_2(t)$ only cross on isolated
points (see \cite{CF-p}). However, if
$\tilde\lambda_1\not=\tilde\lambda_2$  one may 
have $x_1(t)=x_2(t)$ on intervals of non-empty interior: the condition
$\Gamma\not=0$ prevents this situation from happening. For example, if  
$$V(x)=\begin{pmatrix}{\rm cos} x & {\rm sin} x \\ {\rm sin}x & -{\rm
    cos} x \end{pmatrix} +v(x) {\rm Id}$$ 
with $v$ smooth and at most quadratic, we have $\lambda_1(x)=v(x)-1$
and $\lambda_2(x)=v(x)+1$: classical trajectories for both modes,
issued from the same point of the phase space, are equal. 
\end{remark} 

\subsection{Strategy of the proof of Theorem~\ref{theo:matrix}}
\label{sec:strategy}

The proof is more complicated than in the scalar
case \cite{CF-p}, due to the fact that the spectral projectors do not
commute with the Laplace operator.  From this perspective, a much
finer geometric understanding is needed and we revisit \cite{Hag94,
  GMMP, HJ98,SpohnTeufel,Teufel} by adapting to our nonlinear context
ideas contained therein.

Observe first that the function $\varphi^\eps$ satisfies
\begin{equation}\label{eq:phieps}
i\eps\partial_t \varphi^\eps+{\eps^2\over 2} \d_x^2 \varphi^\eps
= \Tc_\eps(t,x)\varphi^\eps 
+\Lambda\,\eps^{3/2}|\varphi^\eps|_{\C^N}^2\varphi^\eps ,
\end{equation}
where 
$$\Tc_\eps(t,x)= \lambda_1(x(t))+\lambda'_1(x(t))(x-x(t))+{1\over 2}
\lambda''(x(t))(x-x(t))^2.$$ 
This term corresponds to the beginning of the Taylor expansion of
$\l_1$ about $x(t)$.
Therefore,  the function
$w^\eps(t,x)=\psi^\eps(t,x)-\varphi^\eps(t,x) \chi^1(t,x)$
satisfies
 $ w^\eps_{\mid t=0}=r_0^\eps$ and 
$$i\eps\partial_tw_\eps (t,x)+{\eps^2\over 2} \d_x^2  
w^\eps(t,x) -V(x) w^\eps(t,x)=\eps \widetilde{NL}^\eps
(t,x)+\eps \widetilde{L}^\eps(t,x)$$ 
where 
\begin{align*}
\widetilde{NL}^\eps & =  \Lambda \, \eps^{1/2}\left(
  \left\lvert\varphi^\eps\chi^1+ w^\eps\right\rvert_{\C^N}^2
(\varphi^\eps\chi^1+
  w^\eps) -|\varphi^\eps |^2\varphi^\eps\chi^1\right),\\ 
\widetilde{L}^\eps & =  i\partial_t\chi^1 \varphi^\eps
+\eps\partial_x\chi^1\partial_x \varphi^\eps+{\eps\over 2}
\varphi^\eps \partial_x^2 \chi^1 +\eps^{-1}\(\lambda_1(x)- \Tc_\eps\)
\varphi^\eps\chi^1. 
\end{align*}
Since $\varphi^\eps$ is
concentrated near $x=x(t)$ at scale $\sqrt\eps$,   we have
$$(\lambda_1(x)-\Tc_\eps)\varphi^\eps=\O\(\eps^{3/2}{\rm
  e}^{Ct}\)\;\;{\rm in}\;\;L^2(\R),$$
where we have used Theorem~\ref{theo:profil}.  
The term $\widetilde{L}^\eps$ \emph{a priori} presents  an $\O(1)$
contribution, which is an obstruction to infer that $ w^\eps$ is
small by applying Gronwall
Lemma. Observing that in view of the estimates on the classical flow
(see~\eqref{expgrowth}) 
$$\eps\partial_x \varphi^\eps =i\xi(t)\varphi^\eps
+\O\left(\sqrt\eps\,{\rm e}^{Ct}\right)\;\;\text{ in}\;\;L^2(\R),$$ 
we write,  
$$\widetilde L^\eps=i\left(\partial_t \chi^1 +\xi(t)\partial_x
  \chi^1\right)\varphi^\eps +\O\left(\sqrt\eps\,{\rm
    e}^{Ct}\right)\;\;{\rm in} \;\;L^2(\R).$$ 
The choice of the time-dependent eigenvectors ensures that for all
time,  the $\O(1)$ contribution of $\widetilde{L}^\eps$ is orthogonal
to the first mode (the eigenspace associated with $\lambda^1$).
Then, to get rid of these terms, we introduce a correction term to
$w^\eps$. We set  
$$\theta^\eps(t,x)=w^\eps(t,x)+\eps g^\eps(t,x),\;\;
g^\eps(t,x)=\sum_{2\le j\le P} \,\sum_{1\le \ell\le d_j}
g^\eps_{j,\ell}(t,x)\chi_j^\ell(x),$$ 
 where for $j\ge 2$ and for $1\le \ell\le d_j$, the function
 $g_{j,\ell}^\eps(t,x)$ solves the scalar Schr\"odinger equation 
  \begin{equation}\label{eq:gk}
     i\eps\partial_t g_{j,\ell}^\eps +{\eps^2\over 2} \d_x^2 
     g_{j,\ell}^\eps-\lambda_j(x)g_{j,\ell}^\eps=\varphi^\eps
     r_{j,\ell}\quad ;\quad
  g_{j,\ell\mid t=0}^\eps=0,
  \end{equation}
  where 
  \begin{equation}\label{def:rkl}
  r_{j,\ell}(t,x)=-i\left(\partial_t
    \chi^1(t,x)+\xi(t)\partial_x\chi^1(t,x)\;,\;
    \chi_j^\ell(x)\right)_{\C^N}. 
  \end{equation}
 The function $\theta^\eps(t)$ then solves 
 \begin{equation}\label{eq:weps}
 \left\{ 
   \begin{aligned}
      i\eps\partial_t \theta^\eps(t,x) +{\eps^2\over 2} \d_x^2   \theta^\eps(t,x)
      &=V(x) \theta^\eps(t,x)+\eps NL^\eps (t,x)+\eps L^\eps(t,x), 
 \\
 \theta^\eps_{\mid t=0}&=r_0^\eps,
   \end{aligned}
\right.
 \end{equation}
 with 
 \begin{align}\label{def:NLeps}
 NL^\eps & =  \Lambda \,
 \eps^{1/2}\left(|\varphi^\eps\chi^1+\theta^\eps-\eps g^\eps|_{\C^N}^2 
(\varphi^\eps \chi^1+\theta^\eps-\eps g^\eps)
-|\varphi^\eps |^2\varphi^\eps\chi^1\right) ,\\ 
   \label{def:Leps}
 L^\eps  & = 
   \widetilde L^\eps +\left(i\eps\partial_t+{\eps^2\over 2}\d_x^2 
    -V(x)\right)g^\eps(t,x)\\
    \nonumber 
  &  =  
  \O(\sqrt\eps{\rm e}^{Ct})+\sum_{2\le j\le P} \,\sum_{1\le \ell\le
    d_j} \left[
    {\eps^2\over 2} \d_x^2 ,\chi_j^\ell\right] g^\eps_{j,\ell}
 \end{align}
 where the $\O(\sqrt\eps{\rm e}^{Ct})$ holds in $L^2 $.
  The proof of the theorem then follows from a precise control of the
  functions $\chi_j^\ell$ and $g^\eps_{j,\ell}$, which is achieved in
  Sections~\ref{sec:eigenvector} and \ref{sec:correctionterms},
  respectively. Then, the analysis of  $ 
   \theta^\eps$ as~$\eps$ goes to zero by an energy method  is
   presented in Section~\ref{sec:consistency}.  

\section{The family of time-dependent eigenvectors}
\label{sec:eigenvector}

In this section we prove Proposition~\ref{prop:timedepeigen},
recalling the construction of the eigenvectors
satisfying~\eqref{eq:eigenvector}, and analyzing the behavior of their
derivatives for large time. We follow the proof of \cite{Hag94}. 
More generally, we prove the following result which implies
Proposition~\ref{prop:timedepeigen}.  We consider the Hamiltonian
curves of ${1\over 
  2}|\xi|^2+\lambda_j(x)$, 
that we denote by $\left(x_j(t),\xi_j(t)\right)$.

 \begin{proposition} \label{prop:timedepeigenbis}
 There exists a smooth orthonormal basis of $\C^N$
 $\left(\chi^{\ell}_j(t,x)\right)_{\scriptstyle{1\le\ell\le
     d_j}\atop\scriptstyle{1\le j\le P}}$  such that  for 
 all $t$,
 $\left(\chi^{\ell}_j(t,x)\right)_{1\le\ell\le d_j}$ spans the
 eigenspace associated to $\lambda_j$,  with
 $\chi^1(0,x)=\chi(x)$ and for $m\in\{1,\cdots, d_j\}$,
$$
\left(\chi^m_j(t,x),\partial_t \chi^\ell_j(t,x) + \xi_j(t)\partial_x \chi
  ^\ell_j(t,x)\right)_{\C^N}=0. 
$$
Moreover, 
for $\ell\in\{1,\dots,d_j\}$, $k,p\in\N$, there exists a constant
$C=C(p,k)$ such that  
$$
\left\lvert \partial_t^p \partial_x^k\chi ^\ell_j
  (t,x)\right\rvert_{\C^N}\le C \, {\rm 
  e}^{C t} \<x\>^{(k+p)(1+n_0)}, $$
where $n_0$ appears in \eqref{gapcondition}. 
\end{proposition}

\begin{proof}[Proof of Proposition~\ref{prop:timedepeigenbis}]
We consider a smooth basis of eigenvectors
$(\chi_j^\ell(0))_{\scriptstyle{1\le\ell\le d_j}\atop\scriptstyle{1\le
    j\le P}}$ such that $\chi_1^1(0)=\chi$. Then, we denote by
$\Pi_j(x)$ the smooth eigenprojector associated with the 
eigenvalue $\lambda_j(x)$ and define 
$$K_j(x)=-i\left[\Pi_j(x),\partial_x\Pi_j(x)\right].$$
 We set $z=x-x_j(t)$
and we consider  the Schr\"odinger type equation 
\begin{equation}\label{def:Y}
i\partial_t Y_j^\ell(t,z)=\xi_j(t)
K_j(z+x_j(t))Y_j^\ell(t,z)\quad ;\quad
Y_j^\ell(0,z)=\chi_j^\ell(x_j(0)+z).
\end{equation}
Let us prove that the  vector $Y_j^\ell(t,z)$ is in the eigenspace of
$\lambda_j(x_j(t)+z)$. Indeed, 
the evolution of $Z_j^\ell(t,z)=\left({\rm
    Id}-\Pi_j(x_j(t)+z)\right)Y_j^\ell(t,z)$ obeys to  
$Z_j^\ell(0,z)=0$ and 
\begin{align*}
\partial_t Z_j^\ell(t,z)& =  -\xi_j(t)\partial_x\Pi_j
(x_j(t)+z)Y_j^\ell \\ 
&\quad  -\xi_j(t) ({\rm Id} -
\Pi_j(x_j(t)+z))[\Pi_j(x_j(t)+z),\partial_x\Pi_j(x_j(t)+z)] Y_j^\ell\\ 
& =  -\xi_j(t) \partial_x\Pi_j(x_j(t)+z) ({\rm Id}-\Pi_j(x_j(t)+z))Y_j^\ell\\
&  =  -\xi_j(t) \partial_x\Pi_j (x_j(t)+z)Z_j^\ell
\end{align*}
where we have used $\partial_x
\Pi_j=\partial_x(\Pi_j^2)=\Pi_j\partial_x\Pi_j+\(\partial_x\Pi_j\)\Pi_j,$
whence 
\begin{equation*}
  \Pi_j (\partial_x\Pi_j )\Pi_j =\Pi_j \(\Pi_j \partial_x\Pi_j 
+\(\partial_x\Pi_j \)\Pi_j \)\Pi_j 
=2  \Pi_j  (\d_x \Pi_j ) \Pi_j =0.
\end{equation*}
Therefore, $Z_j^\ell(t)$ satisfies an equation of the form $\partial_t
Z_j^\ell=A(t,z)Z_j^\ell$, which combined with $Z_j^\ell(0)=0$, implies
$Z_j^\ell(t)=0$ for all $t\in\R$: the vectors $Y_j^\ell(t,z)$ are
eigenvectors of $V(x_j(t)+z)$ for the eigenvalue
$\lambda_j(x_j(t)+z)$.\\ 
  Besides, since $\xi_j(t)
K_j(z+x_j(t))$ is self-adjoint, $Y_j^\ell(t,z)$ is normalized for all
$t$,  and  
the family  $(Y_j^\ell)_{1\le\ell\le d_j}$  is orthonormal.   
We define $\chi_j^\ell(t,x)$ by
\begin{equation}\label{def:chi}
\chi_j^\ell(t,x)=Y_j^\ell(t, x-x_j(t))
\end{equation}
and 
we obtain an orthonormal basis of eigenvectors of $V(x)$.\\
 It remains to check that~\eqref{eq:eigenvector} holds.
We have 
\begin{eqnarray*}
\partial_t \chi_j^\ell+\xi_j(t)\partial_x \chi_j^\ell & = & 
\partial_tY_j^\ell (t,x-x_j(t))\\
& = & i
\xi_j(t)K_j(x)\chi_j^\ell\\ &=& -\xi_j(t)[\Pi_j(x),\partial_x\Pi_j(x)]\chi_j^\ell,
\end{eqnarray*} 
whence
\begin{eqnarray*}
\left(\partial_t \chi_j^\ell+\xi_j(t)\partial_x
  \chi_j^\ell\;,\;\chi_j^k\right)_{\C^N} & = & - 
\xi_j(t) \left([\Pi_j,\partial_x\Pi_j] \chi_j^\ell,\chi_j^k\right)_{\C^N}\\
& = & -\xi_j(t) \left(\Pi_j[\Pi_j,\partial_x\Pi_j]
  \Pi_j\chi_j^\ell,\chi_j^k\right)_{\C^N} 
\end{eqnarray*}
since $\chi_j^{\ell/k}=\Pi_j\chi_j^{\ell/k}$.
We then observe that $\Pi_j^2=\Pi_j $ implies
$$\Pi_j[\Pi_j,\partial_x\Pi_j] \Pi_j=\Pi_j^2 \partial_x \Pi_j\Pi_j
-\Pi_j\partial_x\Pi_j\Pi_j^2=0.$$ 
This concludes the first part of
Proposition~\ref{prop:timedepeigen}. It remains to study the behavior
at infinity of the vectors $\chi_j^\ell(t,x)$ and of their
derivatives. 

\smallbreak

\noindent By the definition of $\chi_j^\ell(t,x)$ in~\eqref{def:chi},
it is enough to prove  
\begin{eqnarray*}
| \partial_t^p\partial_x^k Y_j^\ell (t,z)|_{\C^N}& \lesssim & {\rm
  e}^{Ct} \<x_j(t)+z\>^{(p+k)(1+n_0)}. 
\end{eqnarray*}
For this, 
we crucially use the estimates of Lemma~\ref{lem:projectors} and we
argue by induction.   
Let us first consider the case $p=1$ and 
$k=0$. 
By Lemma~\ref{lem:projectors}, we have $|K_j(x)|\lesssim
\<x\>^{1+n_0}$, whence~\eqref{def:Y} gives 
$$|\partial_t Y_j^\ell(t,x-x_j(t))|\lesssim
|\xi_j(t)||K_j(x)|\lesssim{\rm e}^{Ct} \<x\>^{1+n_0}.$$ 
Let us now suppose $k\ge 1$ and $p=0$. We observe that
$\partial_z^k Y_j^\ell(t,z)$ solves 
\begin{equation}\label{eq:dbetaY}
\left\{
  \begin{aligned}
  i \partial_t \partial_z^k Y_j^\ell (t,z)& =-i \xi_j(t)
K_j(z+x_j(t))\partial_z^k Y_j^\ell(t,z) +f(t,z),\\
\partial_z^k Y_j(0,z)&=\partial_x^k \chi_j^\ell (0, z+x_j(0)), 
  \end{aligned}
\right.
\end{equation}
where 
$$f(t,z)=\sum_{0\le\gamma\le k-1}c_\gamma\, \xi_j(t) \partial^\gamma_z
K_j(z+x_j(t))\partial_z^{k-\gamma}Y_j^\ell(t,z),$$ 
for some complex numbers 
$c_\gamma$ independent of $t$ and $z$. We obtain
$$\partial_z^k Y_j^\ell (t,z)={\mathcal U}_j(t,0) \partial_x^k
\chi_j^\ell(z+x_j(0))+\int_0^t {\mathcal U}_j(t,s)f(s,z)ds ,$$
where $ {\mathcal U}_j(t,s)$ denotes the unitary propagator associated
to \eqref{def:Y} (when the initial time is equal to $s$). 
We have  by Lemma~\ref{lem:projectors}
$$|\partial_z^k Y^\ell_j(0,z)|_{\C^N} = \left|\partial_x^k \chi_j^\ell
  (0, z+x_j(0))\right|_{\C^N}\lesssim\<z+x_j(0)\>^{k(1+n_0)},$$ 
therefore the induction assumption 
$$\forall \gamma\in \{0,\dots,k-1\},\;\; | \partial_x^\gamma Y_j^\ell
(t,z)|_{\C^N} \lesssim  {\rm e}^{Ct} \<x_j(t)+z\>^{\gamma(1+n_0)}$$ 
implies, along with Lemma~\ref{lem:projectors},
$$| \partial_x^k Y_j^\ell (t,z)|_{\C^N} \lesssim  {\rm e}^{Ct}
\<x_j(t)+z\>^{k(1+n_0)}.$$ 
We have obtained the estimate for $p=0$, $k\in \N$, and for $p=1$,
$k=0$. Note that Equation~\eqref{eq:dbetaY} yields 
$$\forall k\in\N,\;\; | \partial_t \partial_x^k Y_j^\ell
(t,z)|_{\C^N} \lesssim  {\rm e}^{Ct}
\<x_j(t)+z\>^{(1+k)(1+n_0)},$$ 
and allows to prove the general estimate for time derivatives by an
induction argument which crucially uses the fact that we have an exponential
control of the derivatives in time of~$\xi_j(t)$. This property follows
by induction from \eqref{eq:traj}, \eqref{expgrowth}, and the fact
that $\l_j$ is at most quadratic. 
\end{proof}

Before concluding this section, note that 
in view of the definition of the function $r_{j,\ell}$ in~(\ref{def:rkl}), 
Proposition~\ref{prop:timedepeigen} gives the following corollary.

\begin{corollary}\label{cor:rk} 
For all $p\in\N$ and $k\in \N$, there exists a constant $C=C(p,k)$ such that,
for $x\in\R$, $j\in\{1,\cdots, P\}$ and $\ell\in\{1,\cdots,d_j\}$,
$$\left|\partial_t^p\partial_x^k r_{j,\ell}(t,x)\right|\lesssim{\rm
  e}^{Ct} \< x \>^{(1+p+k)(1+n_0)}.$$ 
\end{corollary}


\section{Analysis of the correction terms}\label{sec:correctionterms}

In this section, we will make use of the 
following norms defined for $p\in\N$,
$$\| f\|_{\Sigma_\eps^p}=\sup_{\alpha+\beta\le p} \left\lVert \lvert x\rvert^\alpha
\eps^\beta f^{(\beta)}(x)\right\rVert_{L^2}.$$ 
We associate  with this norm the 
functional space $\Sigma_\eps^p$ defined by
$$\Sigma_\eps^p=\{f\in L^2(\R^d),\;\;\
\|f\|_{\Sigma_\eps^p}<\infty\}.$$
In view of  \eqref{expgrowth} and \eqref{Linftybound}, for all $p\in
\N$, there exists $c(p)$ such that
\begin{equation}
  \label{eq:phiSigma}
  \|\varphi^\eps(t)\|_{\Sigma^p_\eps}\lesssim {\rm e}^{c(p)t},\quad
  \forall t\ge 0.
\end{equation}
We can obviously take $c(0)=0$ by conservation of the $L^2$-norm, but
in general, the norm of $\varphi^\eps$ in $\Sigma_\eps^1$  potentially
grows exponentially in time (see \cite{Ca-p}).   
 We denote by $U^\eps_k(t)$  the semi-group associated with the
 operator $-{\eps^2\over 2} \d_x^2  +\lambda_k(x)$ and we observe that
 for $p\in\N$, there exists a constant $C(p)$ such that  
 \begin{equation}\label{est:USigma}
 \|U^\eps_k(t)\|_{{\mathcal L}(\Sigma^p_\eps)}\le C(p){\rm e}^{C(p)|t|}.
 \end{equation}
  The following  averaging lemma shows an asymptotic
  orthogonality property.
  \begin{lemma}\label{lem:orth}
 For $T>0$ and $k\not=j$,
 there exists a constant $C$ such that 
 $$\forall t\in [0,T],\;\;\forall p\in\N,\;\;\;\;\left\|{1\over i\eps}
   \int_0^tU_k^\eps(-s)U_j^\eps(s) ds\right\|_{{\mathcal
     L}\(\Sigma_\eps^{(p+3)n_0+p+2}, \Sigma_\eps^p\)}\le C{\rm
   e}^{Ct}.$$ 
 \end{lemma}
 
 \begin{proof}
 We first  observe that
 \begin{equation}\label{eq:partialtUU}
i\eps\partial_t\left(
  U_k^\eps(-t)U_j^\eps(t)\right)=U^\eps_k(-t)\left(
  \lambda_j(x)-\lambda_k(x)\right)U^\eps_j(t). 
\end{equation}
 Indeed, if $f\in L^2(\R)$ and $f^\eps(t)= U_k^\eps(-t)U_j^\eps(t)f.$
 We have 
 \begin{eqnarray*}
 i\eps\partial_t f^\eps(t,x) & = &- \left( -{\eps^2\over 2} \d_x^2 
   +\lambda_k(x)\right)f^\eps(t) + U^\eps_k(-t)\left( -{\eps^2\over 2}
   \d_x^2  +\lambda_j(x)\right)U^\eps_j(t)f\\ 
 & = & U^\eps_k(-t)\left( \lambda_j(x)-\lambda_k(x)\right)U^\eps_j(t)f
 \end{eqnarray*}
because $U^\eps_k(-t)$ commutes with $ -{\eps^2\over 2} \d_x^2 
   +\lambda_k(x)$.
We use Equation~\eqref{eq:partialtUU} to perform an integration by parts:
\begin{eqnarray*}
U_k^\eps(-t)U_j^\eps(t) & = &
U_k^\eps(-t)\left(\lambda_j-\lambda_k\right)^{-1} U_k^\eps(t) U_k^\eps
(-t)\left(\lambda_j-\lambda_k\right)U_j^\eps(t)\\ 
& = & i\eps \,U_k^\eps(-t)\left(\lambda_j-\lambda_k\right)^{-1}
U_k^\eps(t)  \;\partial_t\left(U_k^\eps(-t)U_j^\eps(t) \right). 
\end{eqnarray*}
Therefore,
\begin{eqnarray*}
{1\over i \eps}\int_0^t U_k^\eps(-s) U_j^\eps(s) ds & = &
\left[U_k^\eps(-s)\left(\lambda_j-\lambda_k\right)^{-1}
  U_j^\eps(s)\right]_0^t \\ 
& & -
\int_0^t \partial_s\left(U_k^\eps(-s)\left(\lambda_j-\lambda_k\right)^{-1}
  U_k^\eps(s) \right) U^\eps_k(-s)U^\eps_j(s)\,ds. 
\end{eqnarray*}
Set 
\begin{equation}\label{def:gamma}
\gamma_{j,k}=\left(\lambda_k-\lambda_j\right)^{-1}.
\end{equation}
The behavior as $x$ goes to infinity of these functions is studied in
Appendix~\ref{sec:growthcond} (see Lemma~\ref{lem:gap}). It is proven
there that for all $\beta\in\N$, 
$$\left| \partial_x^\beta \gamma_{j,k}(x) \right|  \lesssim
\<x\>^{n_0+|\beta|(1+n_0)}.$$ 
Since the propagators $U^\eps_k(t)$ and $U^\eps_j(t)$ map continuously
$\Sigma_\eps^p$ into itself uniformly with respect to $\eps$,  we have    
$$\left\| \left[U_k^\eps(-s)\gamma_{j,k}U_j^\eps(s)\right]_0^t
\right\|_{{\mathcal
    L}\(\Sigma_\eps^{(p+1)n_0+p},\Sigma_\eps^p\)}\lesssim C(p),$$ 
where in all this paragraph,  $C(p)$ denotes a generic constant
depending only on the parameter $p\in\N$. 
Besides, we observe that
$$\partial_s\left(U_k^\eps(-s)\gamma_{j,k}U_k^\eps(s) \right)= {1\over
  i\eps} U_k^\eps(-s)\left[-{\eps^2\over 2}
  \d_x^2+\lambda_k\;,\;\gamma_{j,k}\right]U_k^\eps(s) .$$ 
In view of 
$${1\over i\eps} \left[-{\eps^2\over 2}
  \d_x^2+\lambda_k\;,\;\gamma_{j,k}\right] =  
{1\over i\eps} \left[-{\eps^2\over 2} \d_x^2\;,\;\gamma_{j,k}\right]=
i \gamma_{j,k}'(x) \eps\partial_x +i \eps \gamma_{j,k}''(x), $$
and of 
\begin{align*}
&\left\| U_k^\eps(-s)\gamma_{j,k}'(x) \eps\partial_x
  U_j^\eps(s)\right\|_{{\mathcal
    L}\(\Sigma_\eps^{(p+3)n_0+p+2},\Sigma_\eps^p\)}\\
 +&\left\|
  U_k^\eps(-s)\gamma_{j,k}''(x) U_j^\eps(s)\right\|_{{\mathcal
    L}\(\Sigma_\eps^{(p+2)n_0+p+2},\Sigma_\eps^p\)}
  \lesssim  {\rm e}^{Cs}, 
\end{align*}
which comes from~\eqref{est:USigma} and Lemma~\ref{lem:gap}, 
we get 
$$\left\|  \partial_s\left(U_k^\eps(-s)\gamma_{j,k}U_k^\eps(s) \right)
\right\|_{{\mathcal L}(\Sigma_\eps^{p+2},\Sigma_\eps^p)}\lesssim  {\rm
  e}^{Ct},$$ 
which concludes the proof. 
 \end{proof}
 
 \smallbreak

 We now prove the following proposition.
 
 \begin{proposition}\label{prop:boundgk} For $p\in\N$, there exists
   $C(p)$ such that for 
   all $j\ge 2$, and all $\ell\in \{1,\dots,d_j\}$,
 $$\|g^\eps_{j,\ell}(t)\|_{\Sigma_\eps^p}\lesssim {\rm
   e}^{C(p)t},\quad \forall t\ge 0,$$ 
where $g^\eps_{j,\ell}$ is defined in \eqref{eq:gk}. 
 \end{proposition}
 
 \begin{proof} We use Duhamel's formula and write 
 $$g^\eps_{j,\ell}(t)={1\over i\eps}\int_0^t
 U_j^\eps(t-s)\left(\varphi^\eps(s)r_{j,\ell}(s)\right) ds.$$ 
 Besides,
if $\tilde
\varphi^\eps_{j,\ell}(t,x)=\varphi^\eps(t,x)r_{j,\ell}(t,x)$, then we
have, 
$$
\left(i\eps\partial_t +{\eps^2\over 2}\d_x^2 
  -\lambda_1(x)\right)\tilde \varphi^\eps_{j,\ell}  
 =  \underbrace{i\eps\partial_t r_{j,\ell}\varphi^\eps
+  r_{j,\ell} \eps^{3/2} |\varphi^\eps|^{2}\varphi^\eps+ {\eps^2\over
  2} \left[\d_x^2 , r_{j,\ell}(t,x)\right]\varphi^\eps}_{ =:   \eps
\widetilde r^\eps(t,x)}. 
$$
 Therefore, we can write 
 $$\tilde\varphi^\eps_{j,\ell}(t)=U^\eps_1(t)
 \tilde\varphi^\eps_{j,\ell}(0)-i\int_0^t U^\eps_1(t-s)\widetilde
 r^\eps(s)ds,$$ 
whence
\begin{align*}
g^\eps_{j,\ell}(t) & ={ 1\over i\eps}\int_0^t U_j^\eps(t-s)U_1^\eps(s)ds
\tilde\varphi^\eps_{j,\ell}(0)- {1\over \eps}\int_0^t 
\int_0^sU_j^\eps(t-s)U_1^\eps(s-\tau)\widetilde r^\eps(\tau)d\tau ds\\ 
& = { 1\over i\eps}\int_0^t U_j^\eps(t-s)U_1^\eps(s)ds
\tilde\varphi^\eps_{j,\ell}(0)- \int_0^t\left[{1\over
    \eps}\int_\tau^t U_j^\eps(t-s) 
U^\eps_1(s-\tau)ds\right]\widetilde r^\eps(\tau) d\tau.
\end{align*}
Lemma~\ref{lem:orth} yields 
\begin{equation*}
\|g^\eps_{j,\ell}(t)\|_{\Sigma_\eps^p}  \lesssim {\rm e}^{Ct}+
 \int_0^t {\rm e}^{C\tau } \| \widetilde r^\eps(\tau)\|_{\Sigma_\eps^q}d\tau,
 \end{equation*}
 with $q=p+2+(p+3)(1+n_0)$.
 Let us now study $\widetilde r^\eps$. We write 
 $\widetilde r^\eps=\widetilde r^\eps_1+\widetilde r^\eps_2$ with 
 $$\widetilde r^\eps_1(t,x)=i\partial_t r_{j,\ell}\varphi^\eps+ {\eps\over
   2} 
\left[\d_x^2 , r_{j,\ell}(t,x)\right]\varphi^\eps.$$
 In view of Corollary~\ref{cor:rk} and of~\eqref{eq:phiSigma}, 
 we have for all $q\in\N$,
$$\|\widetilde r^\eps_1(t)\|_{\Sigma^q_\eps(\R)}\lesssim {\rm e}^{C(q)t}. $$
A very rough estimate yields
\begin{align*}
  \|\widetilde r^\eps_2(t)\|_{\Sigma^q_\eps}&=
\|\sqrt\eps
r_{j,\ell}|\varphi^\eps|^2\varphi^\eps\|_{\Sigma^q_\eps}\lesssim
\sqrt\eps \left\lVert r_{j,\ell}\<x\>^q
  \varphi^\eps\right\rVert_{\Sigma^q_\eps}\left\lVert 
  \<\eps \d_x\>^q \varphi^\eps\right\rVert_{L^\infty}^2 \\
&\lesssim
\sqrt\eps {\rm e}^{Ct}\left\lVert \<x\>^{q+(1+q)(n_0+1)}
  \varphi^\eps\right\rVert_{\Sigma^q_\eps}\left\lVert 
  \<\eps \d_x\>^q \varphi^\eps\right\rVert_{L^\infty}^2,
\end{align*}
where we have used Corollary~\ref{cor:rk}. Now with \eqref{Linftybound} and
\eqref{eq:phiSigma}, we conclude
\begin{equation*}
  \|\widetilde r^\eps_2(t)\|_{\Sigma^q_\eps} \lesssim {\rm e}^{Ct}.
\end{equation*}
This completes the proof of Proposition~\ref{prop:boundgk}.
\end{proof}

 
\section{Consistency}\label{sec:consistency}
 
We now prove Theorem~\ref{theo:matrix}. We go back to
Equation~\eqref{eq:weps}, that we recall: 
$$
 \left\{ 
   \begin{aligned}
      i\eps\partial_t \theta^\eps(t,x) +{\eps^2\over 2} \d_x^2   \theta^\eps(t,x)
      &=V(x) \theta^\eps(t,x)+\eps NL^\eps (t,x)+\eps L^\eps(t,x), 
 \\
 \theta^\eps_{\mid t=0}&=r_0^\eps,
   \end{aligned}
\right.
$$
where $NL^\eps$ and $L^\eps$ are defined in~\eqref{def:NLeps}
and~\eqref{def:Leps}, respectively.  
The standard $L^2$-estimate yields:
$$
\|\theta^\eps(t)\|_{L^2}\le \|r_0^\eps\|_{L^2} +\int_0^t
\(\|NL^\eps(s)\|_{L^2}+\|L^\eps(s)\|_{L^2}\)ds.$$ 
 In view of~\eqref{def:Leps}, Proposition~\ref{prop:timedepeigenbis}
 and Proposition~\ref{prop:boundgk}, we have  
$$\|L^\eps(t)\|_{L^2}\lesssim \sqrt\eps {\rm e}^{Ct}.$$
Besides, we observe
\begin{align*}
\|NL^\eps(t)\|_{L^2} & \lesssim  \sqrt\eps \left\| \left(
    |\varphi^\eps(t)|^2+|\theta^\eps(t)|_{\C^N}^2+\eps^2
    |g^\eps(t)|_{\C^N}^2 \right)
\(\theta^\eps(t)-\eps g^\eps(t)\)\right\|_{L^2}\\  
 \lesssim  \sqrt\eps &\left(\|\varphi^\eps
  (t)\|^2_{L^\infty}+\|\theta^\eps(t)\|^2_{L^\infty}+
\eps^2\|g^\eps(t)\|_{L^\infty}^2\right)  
\(\|\theta^\eps(t)\|_{L^2}+\eps\|g^\eps(t)\|_{L^2}\).
\end{align*}
In view of~\eqref{Linftybound}, we have
$\|\varphi^\eps(t)\|_{L^\infty}\lesssim \eps^{-1/4}{\rm
  e}^{Ct}$. On the other hand, Proposition~\ref{prop:boundgk} implies,
in view of the 
Gagliardo-Nirenberg inequality
\begin{equation}
  \label{eq:GN}
  \|f\|_{L^\infty} \lesssim
\eps^{-1/2}\|f\|_{L^2}^{1/2}
\|\eps\partial_xf \|_{L^2}^{1/2},
\end{equation}
the estimate
\begin{equation*}
 \eps^2\|g^\eps(t)\|_{L^\infty}^2 \lesssim \eps  {\rm
  e}^{Ct}.
\end{equation*}
Therefore, it is natural to perform a bootstrap argument
assuming, say 
\begin{equation}\label{bootstrap}
\|\theta^\eps(t)\|_{L^\infty}\le \eps^{-1/4}{\rm e}^{Ct}.
\end{equation}
Note that we fixed the value of the constant in factor of the
right hand side equal to one. We did so because $\theta^\eps$, as an
error term, is expected to be  smaller than $\varphi^\eps$ (the approximate
solution) in the limit $\eps\to 0$. 
As long as~\eqref{bootstrap} holds, the $L^2$-estimate implies, in
view of \eqref{eq:erreur-init}
$$\|\theta^\eps(t)\|_{L^2}\lesssim \eps^\kappa+ 
\int_0^t \(\sqrt\eps {\rm e}^{Cs} +{\rm
  e}^{Cs}\|\theta^\eps(s)\|_{L^2} \)ds .$$ 
By Gronwall Lemma, we obtain 
\begin{equation}\label{firstestapriori}
\|\theta^\eps(t)\|_{L^2}\le C\(\eps^\kappa+\sqrt\eps\) {\rm e}^{{\rm e}^{Ct}}.
\end{equation}
It remains to check how long  the bootstrap
assumption~\eqref{bootstrap} holds. For this, we use
Gagliardo-Nirenberg inequality \eqref{eq:GN}, 
and we look for a control of the norm of $\theta^\eps(t)$ in
$\Sigma^1_\eps$. Differentiating the system~\eqref{eq:weps} with
respect to $x$, we find
\begin{equation*}
  i\eps\partial_t(\eps\partial_x \theta^\eps) +{\eps^2\over 2}
  \d_x^2 (\eps\partial_x  \theta^\eps) 
      =V(x)\eps\partial_x \theta^\eps+ \eps
      V'(x)\theta^\eps+\eps^2\partial_x NL^\eps +\eps^2\partial_x
      L^\eps,  
\end{equation*}
We observe that since $V$ is at most quadratic,
$|V'(x)\theta^\eps|_{\C^N}\lesssim \<x\>|\theta^\eps
|_{\C^N}$. Therefore, in order to obtain a closed system of estimates,
we consider the equation satisfied by $x\theta^\eps$: multiply
\eqref{eq:weps} by $x$,
\begin{equation*}
i\eps\partial_t (x\theta^\eps) +{\eps^2\over 2} \d_x^2  (x \theta^\eps)
      =V(x)(x \theta^\eps)+\eps^2\d_x\theta^\eps +\eps xNL^\eps
      +\eps x L^\eps .
\end{equation*}
By Proposition~\ref{prop:boundgk}, we have 
$$\|xL^\eps(t)\|_{L^2}+\|\eps\partial_x L^\eps(t)\|_{L^2}\lesssim
\sqrt\eps {\rm e}^{Ct}.$$ 
Besides, 
\begin{align*}
|xNL^\eps(t,x)|_{\C^N} & \lesssim
\left(|\phi^\eps(t,x)|^2+|\theta^\eps(t,x)|_{\C^N}^2+\eps^2
  |g^\eps(t,x)|_{\C^N}^2\right)\times \\
&\qquad \times
\(|x\theta^\eps(t,x)|_{\C^N}+\eps |x g^\eps(t,x)|_{\C^N}\),\\ 
|\eps\partial_xNL^\eps(t,x)|_{\C^N} & \lesssim
\left(|\phi^\eps(t,x)|^2+|\theta^\eps(t,x)|_{\C^N}^2+\eps^2
  |g^\eps(t,x)|_{\C^N}^2\right)
|\eps\partial_x\theta^\eps(t,x)|_{\C^N}\\
&\quad +\eps\left(|\phi^\eps(t,x)|^2+|\theta^\eps(t,x)|_{\C^N}^2+\eps^2
  |g^\eps(t,x)|_{\C^N}^2\right)
|\eps\partial_xg^\eps(t,x)|_{\C^N}\\
&\quad 
+\lvert\eps\partial_x\phi^\eps(t,x)\rvert
\times\lvert \phi^\eps(t,x)\rvert
\times\lvert\theta^\eps(t,x)\rvert_{\C^N}\\
& \quad + \eps \lvert \phi^\eps(t,x)\rvert^2
\times |\partial_x\chi^1(t,x)|_{\C^N}\times\lvert\theta^\eps(t,x)\rvert_{\C^N}. 
\end{align*}
Arguing as before and  using again~\eqref{Linftybound}, we obtain that
under~\eqref{assumption} we have  
$$\|\eps\partial_x\theta^\eps(t)\|_{L^2}+\|x\theta^\eps(t)\|_{L^2}\lesssim
\(\eps^\kappa+\sqrt\eps\){\rm e}^{{\rm e}^{Ct}}.$$ 
Gagliardo--Nirenberg inequality then implies
$$\|\theta^\eps(t)\|_{L^\infty}\lesssim \eps^{-1/2}\(\eps^\kappa+\sqrt\eps\) {\rm
  e}^{{\rm e}^{Ct}}.$$ 
We infer that \eqref{bootstrap} holds (at least) as long as
\begin{equation*}
  \(\eps^{\kappa-1/2}+1\) {\rm
  e}^{{\rm e}^{Ct}}\ll \eps^{-1/4}{\rm e}^{Ct},
\end{equation*}
which is ensured provided that
 $t\le C {\log}{\log}\left({1\over\eps}\right)$, for some suitable
 constant $C$, since $\kappa>1/4$. This concludes the bootstrap
 argument: we infer 
$$\sup_{|t|\le C{\log}{\log}\left(\frac{1}{\eps}\right)}
\(\|\theta^\eps(t)\|_{L^2}+\|x \theta^\eps(t)\|_{L^2}+ \|\eps\d_x
\theta^\eps(t)\|_{L^2}\)\Tend\eps 0 0 .$$ 
Theorem~\ref{theo:matrix} then follows from the above asymptotics,
 together with the relation  $\theta^\eps = w^\eps +\eps g^\eps$, and
 Proposition~\ref{prop:boundgk}.  

 \begin{remark}\label{rem:whylog}
   In the case where $V^\eps = D+\eps W$, as in Remark~\ref{rem:log},
   the proof can be adapted, in order to reproduce the argument given
   in \cite{CF-p}. The main point to notice is that (local in time)
   Strichartz estimates are available for the propagator associated to
   $-\frac{\eps^2}{2}\d_x^2+D(x)$, thanks to \cite{Fujiwara}. Then in the
   presence of the power $\eps$ in front of $W$, the potential $\eps
   W$ can be considered as a source term in the error estimates: the
   factor $\eps$ is crucial to avoid a singular power of $\eps$ due to
   the presence of $\eps$ in front of the time derivative in
   \eqref{eq:weps}. The proof in \cite[Section~6]{CF-p} for the cubic,
   one-dimensional Schr\"odinger equation can be reproduced: another 
   bootstrap argument can be invoked, which does not involve
   Gagliardo--Nirenberg inequalities, since a useful \emph{a priori}
   estimate for the envelope $u$ is available. 
 \end{remark}

\section{Superposition}\label{sec:sup}

As explained in the introduction, the only difficulty in the proof of
Theorem~\ref{theo:superposition1} is to treat a nonlinear
interaction term. Indeed, we set 
$$w^\eps=\psi^\eps-\varphi^\eps_1\chi^1-\varphi^\eps_2\chi^2+\eps g^\eps$$
where $g^\eps$ is the sum of two correction terms, similar to the one
introduced in \S\ref{sec:strategy}. More precisely, set $p(1)=1$, and 
$p(2)=1$ if $\tilde\lambda_1=\tilde\lambda_2$, $p(2)=2$ otherwise.
Define
$g^\eps=g^\eps_1+g^\eps_2$, with
\begin{align*}
g^\eps_1 & =  \sum_{1\le j\le P, \,j\not=p(1)} \,\sum_{1\le \ell\le d_j}
g^\eps_{j,1,\ell}(t,x)\chi_j^\ell(t,x),\\
g^\eps_2 & =    \sum_{1\le j\le P,\,j\not=p(2)} \,\sum_{1\le \ell\le d_j}
g^\eps_{j,2,\ell}(t,x)\chi_j^\ell(t,x) ,
\end{align*}
 where
 for $ k=\{1,2\}$, $j\not=p(k)$ and  $1\le \ell\le d_j$, the function
 $g_{j,k,\ell}^\eps(t,x)$ solves the scalar Schr\"odinger equation
  \begin{equation}\label{eq:gjk}
     i\eps\partial_t g_{j,k,\ell}^\eps +{\eps^2\over 2} \d_x^2 
     g_{j,k,\ell}^\eps-\lambda_j(x)g_{j,k,\ell}^\eps=\varphi^\eps
     r_{j,k,\ell}\quad ;\quad
  g_{j,k,\ell\mid t=0}^\eps=0,
  \end{equation}
  where 
  \begin{equation}\label{def:rjkl}
  r_{j,k,\ell}(t,x)=-i\left(\partial_t
    \chi^{k}(t,x)+\xi_{p(k)}(t)\partial_x\chi^{k}(t,x)\;,\;
    \chi_j^\ell(t,x)\right)_{\C^N}. 
  \end{equation}
 The function $w^\eps(t)$ then solves 
 \begin{equation*}
      i\eps\partial_t w^\eps +{\eps^2\over 2} \d_x^2   w^\eps
      =V(x) w^\eps+\eps NL^\eps +\eps L^\eps\quad
      ;\quad  w^\eps_{\mid t=0}=0,
 \end{equation*}
 with 
 \begin{equation*}
 L^\eps     = 
  \O(\sqrt\eps{\rm e}^{Ct})+\sum_{k=1,2}\sum_{{1\le j\le P}\atop{j\not=p(k)}}
  \,\sum_{1\le \ell\le
    d_j} \left[
    {\eps^2\over 2} \d_x^2 ,\chi_j^\ell\right] g^\eps_{j,k,\ell}
    =  \O(\sqrt\eps{\rm e}^{Ct}).
 \end{equation*}
Here, the $\O(\sqrt\eps{\rm e}^{Ct})$ holds in $\Sigma_\eps^1$, from
Proposition~\ref{prop:boundgk}.  
Besides, 
\begin{align*}
NL^\eps & =  \sqrt\eps \Bigl( \left| w^\eps +
  \varphi^\eps_1\chi^1+\varphi^\eps_2\chi^2+\eps
  g^\eps\right|^2\left(w^\eps +
  \varphi^\eps_1\chi^1+\varphi^\eps_2\chi^2+\eps g_\eps\right) \\
&  \qquad\qquad
 -  |\varphi^\eps_1|^2\varphi^\eps_1\chi^1
-|\varphi^\eps_2|^2\varphi^\eps_2\chi^2\Bigr) 
\end{align*}
Adding and subtracting the term $\sqrt\eps | 
  \varphi^\eps_1\chi^1+\varphi^\eps_2\chi^2|^2(
  \varphi^\eps_1\chi^1+\varphi^\eps_2\chi^2) $, we have
$$|NL^\eps|\le N^\eps_S+N^\eps_I ,$$
where we have the pointwise estimates
\begin{align*}
N^\eps_I & \lesssim  \sqrt\eps \left(
  |\varphi^\eps_1|^2|\varphi^\eps_2|+|\varphi^\eps_2|^2 
|\varphi^\eps_1|\right),\\ 
N^\eps_S & \lesssim  \sqrt\eps \left( |\varphi^\eps_1|^2+
|\varphi^\eps_2|^2+|w^\eps|^2+\eps^2|g^\eps|^2\right)
\left( |w^\eps|+\eps|g^\eps|\right).
\end{align*}
The semilinear term $N^\eps_S$ can be treated exactly in the same
manner as in Section~\ref{sec:consistency}. It remains to analyze
$\displaystyle{\int_0^t\left\|NL^\eps_I(s)\right\|_{\Sigma_\eps^1}ds.}$ 
We observe 
\begin{equation*}
  \sqrt\eps\int_0^t \left\| |\varphi^\eps_1(s)|^2
  \varphi^\eps_2(s)\right\|_{L^2} ds =
\int_0^t \left\| \left|
    u_1\left(s,y-{x_1(s)-x_2(s)\over\sqrt\eps}\right)\right|^2 
    u_2(s,y) \right\|_{L^2} ds,
\end{equation*}
and we note that 
the contribution of $|\varphi^\eps_1|^2\varphi^\eps_2$ and that of
$|\varphi^\eps_2|^2\varphi^\eps_1$ play the same role. Also, we leave
out the other terms which are needed in view of a $\Sigma_\eps^1$
estimate, since they create no trouble. 
 Arguing as in  
  \cite[Lemma~6.1]{CF-p}, we obtain:

\begin{lemma}
Let $T\in\R$, $0<\gamma<1/2$ and 
$$I^\eps(T)=\left\{ t\in[0,T],\;\;\left|x_1(t)-x_2(t)\right| \le
  \eps^\gamma\right\}.$$ 
Then, for all $k\in\N$, there exists a constant $C_k$ such that 
$$\int_0^T \| NL^\eps_I(t)\|_{\Sigma_\eps^1} dt \lesssim
(M_{k+2}(T))^3 \left( T\eps^{k(1/2-\gamma)}+|I^\eps(T)|\right){\rm
  e}^{C_kT},$$ 
with 
$$M_{k}(T)=\sup\left\{ \|\<x\>^\alpha\partial_x ^\beta
  u_j\|_{L^\infty([0,T],L^2(\R))};\ 
  j\in\{1,2\},\quad \alpha+\beta\le k\right\}.$$ 
\end{lemma}
In view of this lemma and of Equation~\eqref{boundedmoments}, we obtain 
$$\int_0^T \| NL^\eps_I(t)\|_{\Sigma_\eps^1} dt \lesssim {\rm e}^{CT}
\left( T\eps^{k(1/2-\gamma)}+|I^\eps(T)|\right),$$ 
and the next lemma yields the conclusion. 

\begin{lemma} Set 
$$\Gamma=\inf_{x\in \R} \left\lvert
  \tilde\lambda_1(x)-\tilde\lambda_2(x)-(E_1-E_2)\right\rvert
,$$ 
and suppose $\Gamma>0$. Then for $0<\gamma<1/2$, there exists
$C_0,C_1>0$ such that  
$$|I^\eps(t)|\lesssim \eps^\gamma \Gamma^{-2} {\rm e}^{C_0 t},\quad 0\le
t\le C_1 {\rm log}\left({1\over \eps}\right).$$ 
\end{lemma}

\begin{proof}
Consider
 $J^\eps(t)$  an interval of maximal length included in $I^\eps(t)$, and
 $N^\eps(t)$ the number of such intervals. The result comes from the
 estimate
 $$|I^\eps(t)|\le N^\eps(t) \times \max |J^\eps(t)|, $$
 with 
\begin{equation}\label{J}
 |J^\eps(t)|  \lesssim  \eps^\gamma e^{Ct}
 \Gamma^{-1}\;\;{\rm and}\;\; 
 N^\eps(t)  \lesssim   te^{Ct} \Gamma^{-1}  ,
 \end{equation}
provided that $\eps^\gamma e^{Ct}\ll 1$. 
Let us prove the first property: consider 
$\tau,\sigma\in J^\eps(t)$. There exists  $t^*\in [\tau,\sigma]$ such that 
 $$\left|\left(x_1(\tau)-x_2(\tau)\right)
   -\left(x_1(\sigma)-x_2(\sigma)\right)\right|=
|\tau-\sigma\rvert \left\lvert\xi_1(t^*)-\xi_2(t^*)\right|,$$ 
whence
 $$ |\tau-\sigma|\le |\xi_1(t^*)-\xi_2(t^*)|^{-1} \times 2 \eps^\gamma.$$
 On the other hand,
 \begin{equation*}
  |\xi_1(t^*)-\xi_2(t^*)|\ge \left\lvert \lvert\xi_1(t^*)\rvert-\lvert
    \xi_2(t^*)\rvert\right\rvert
\ge \frac{\left| |\xi_1(t^*)|^2-
|\xi_2(t^*)|^2\right|}{|\xi_1(t^*)|+|\xi_2(t^*)|}.  
 \end{equation*}
We use 
\begin{align*} 
|\xi_1(t^*)|+|\xi_2(t^*)| & \lesssim  e^{Ct},\\
|\xi_1(t^*)|^2-|\xi_2(t^*)|^2 & = 
2\left(E_1-E_2-\tilde\lambda_1(x_1(t^*))+\tilde\lambda_2(x_2(t^*))\right),
\end{align*}
and infer
\begin{align*}
 \left\lvert
   E_1-E_2-\tilde\lambda_1(x_1(t^*))+\tilde\lambda_2(x_2(t^*))\right\rvert
 & \ge  \left\lvert
   E_1-E_2-\tilde\lambda_1(x_1(t^*))+\tilde\lambda_2(x_1(t^*))\right\rvert\\
&\quad - \left\lvert
   \tilde\lambda_2(x_1(t^*))+\tilde\lambda_2(x_2(t^*))\right\rvert\\
& \ge  \Gamma - C\eps^\gamma {\rm e}^{Ct},
\end{align*}
where we have used the fact that $\tilde\lambda_2$ is at most
quadratic. Therefore, if $\eps^\gamma {\rm e}^{Ct}$ is sufficiently
small,
\begin{equation*}
 \left\lvert
   E_1-E_2-\tilde\lambda_1(x_1(t^*))+\tilde\lambda_2(x_2(t^*))\right\rvert
 \ge \frac{\Gamma}{2}. 
\end{equation*}
We infer
 $$ |\tau-\sigma|  \lesssim \eps^\gamma{\rm e}^{Ct} \Gamma^{-1},$$
 provided $\eps^\gamma e^{Ct}\ll 1$.
 \smallbreak

Let us now consider $N^\eps(t)$. We use that as $t$ is large, $N^\eps(t)$ is
comparable to the number of distinct intervals of maximal size where
$|x_1(t)-x_2(t)|\ge \eps^\gamma$.  
More precisely, $N^\eps(t)$ is smaller  than $t$ divided by  the
minimal size of these intervals.  Therefore, we consider  one interval
$]\tau,\sigma[$  of this type and we look for lower bound of
$\sigma-\tau$. We have  
 $$|x_1(\tau)-x_2(\tau)|=|x_1(\sigma)-x_2(\sigma)|=\eps^\gamma, \text{ and
   }\forall t\in[\tau,\sigma],\quad|x_1(t)-x_2(t)|\ge\eps^\gamma.$$
   Besides,  inside $]\tau,\sigma[$, $x_1(t)-x_2(t)$ has a constant
   sign that we can suppose to be~$+$ (one argues similarly if it
   is~$-$). Under this assumption, we have 
$$\xi_1(\tau)-\xi_2(\tau)>0\;\;{\rm and}\;\;\xi_1(\sigma)-\xi_2(\sigma)<0.$$
Using the exponential control of $\lambda_j'(x_j(t))$ for
$j\in\{1,2\}$, we obtain 
\begin{equation}\label{Bebert}
\left(\xi_1(\tau)-\xi_2(\tau)\right)-
\left(\xi_1(\sigma)-\xi_2(\sigma)\right)\lesssim{\rm
  e}^{Ct}(\sigma-\tau).\end{equation}
 We write 
\begin{align}\label{Be}
  \xi_1(\tau)-\xi_2(\tau)=|\xi_1(\tau)-\xi_2(\tau)|&\ge
  \frac{\left||\xi_1(\tau)|^2- |\xi_2(\tau)|^2\right|}{|\xi_1(\tau)|+
    |\xi_2(\tau)|}\\ \nonumber
&\gtrsim
  e^{-Ct}\left||\xi_1(\tau)|^2-|\xi_2(\tau)|^2\right|
\end{align}
and 
\begin{equation}\label{bert}
-\xi_1(\sigma)+\xi_2(\sigma)=|\xi_1(\tau)-\xi_2(\tau)|\gtrsim {\rm
  e}^{-Ct} \left| |\xi_1(\sigma)|^2-|\xi_2(\sigma)|^2\right|. 
\end{equation}
As before, we prove 
\begin{equation*}
  \left||\xi_1(\tau)|^2- |\xi_2(\tau)|^2\right| + \left|
    |\xi_1(\sigma)|^2-|\xi_2(\sigma)|^2\right|\gtrsim  \Gamma , 
\end{equation*}
provided that $\eps^\gamma {\rm e}^{Ct}\ll 1$. 
 Therefore, plugging the latter equation, \eqref{Be} and~\eqref{bert}
 into~\eqref{Bebert}, we obtain 
$$\sigma-\tau\gtrsim e^{-2Ct}\Gamma\;\;{\rm thus } \;\;
N^\eps(t)\lesssim t e^{Ct} \Gamma^{-1}\lesssim {\rm
  e}^{Ct}\Gamma^{-1}$$ 
which completes the proof of Theorem~\ref{theo:superposition1}.
\end{proof}

\begin{remark}
  The proof shows that if the approximation of
  Theorem~\ref{theo:matrix} is proven to be valid on some time
  interval $[0,C\log(1/\eps)]$, then Theorem~\ref{theo:superposition1}
  will also be valid on a time interval of the same form. 
\end{remark}

\appendix
\section{Global existence of the exact solution}
\label{sec:appA}

The proof of Lemma~\ref{lem:existpsi} follows classical
arguments; see \cite{TsutsumiL2} (or \cite{CazCourant}) for more
details. We suppose $\eps=1$  
without loss of generality. We use the decomposition  
$V(x)=D(x)+W(x)$ of Assumption~\ref{assumption} and we denote by
$U(t)$ the unitary propagator of $-{1 \over 2}\d_x^2 +D(x)$.  
Let  $X_T$ be the set
$$X_T=\left\{\psi\in{\mathcal C}(I_T,\Sigma_1^1),\;\psi,x\psi,\nabla
  \psi\in L^8(I_T,L^4(\R,\C^N))\right\},\;\;I_T=]s-T,s+T[$$ 
for $s\in\R$ and $T\in\R$ to be fixed later.
The proof  consists in a fixed point argument for the function
$$\Phi_s: \psi\mapsto \Phi_s(\psi)$$
where for $s\in\R$, the function $\Phi_s(\psi)$ is defined by 
$$\Phi_s(\psi)(t)
=U(t-s)\psi(s)-i\Lambda\int_s^tU(t-\tau)\left(|\psi|^{2}_{\C^N}
  \psi\right)(\tau)d\tau-i\int_s^\tau
U(t-\tau)\left(W\psi(\tau)\right)d\tau.$$ 
By \cite{Fujiwara}, local in time Strichartz estimates are available for
$U$. 
Strichartz estimates and H\"older inequality imply that there exists a
constant $C>0$ such that 
\begin{align*}
  \|\Phi_s(\psi)\|_{L^8(I_T,L^4)\cap L^\infty(I_T,L^2)} & \le C
\|\psi(s)\|_{L^2} + C
\|\psi\|^{2}_{L^{8/3}(I_T,L^4)}\|\psi\|_{L^8(I_T,L^4)}\\
&\quad +
C\|W\psi\|_{L^1(I_T,L^2)}.
\end{align*}
Using the boundedness of the coefficients of $W$ and H\"older
inequality in time, we obtain  
$$\displaylines{\|\Phi_s(\psi)\|_{L^8(I_T,L^4)\cap L^\infty(I_T,L^2)}\le
  C\|\psi(s)\|_{L^2} +C\sqrt T \|\psi\|_{L^8(I_T,L^4)}^3+CT\|\psi\|_
  {L^\infty(I_T,L^2)}.\cr}$$ 
We can then infer that $\Phi_s$ is a contraction on a ball of $X_T$
for some  $T$ which  depends only on $\|\psi(s)\|_{L^2}$. Then, the
conservation of $\|\psi(t)\|_{L^2} $ yields the lemma.

\section{Some formulas involving the projectors}
\label{sec:formulas}

In this section, we list and prove some formulas which will be used in
the course of the computations in the next appendix. We consider here
the more general case $x\in \R^d$, with $d\ge 1$. Fix once and for all
in this paragraph $j\in
\{1,\dots,P\}$ and  $\ell \in \{1,\dots,d\}$.  First, recall
that we have seen in \S\ref{sec:eigenvector} that since $\Pi_j^2=\Pi_j$,
\begin{equation}
  \label{eq:plp}
  \Pi_j\(\d_\ell
\Pi_j\)\Pi_j =0. 
\end{equation}
Differentiating the relation $\Pi_j^2=\Pi_j$, we find: $\forall j\in
\{1,\dots,P\},\ \forall \ell \in \{1,\dots,d\}$,
\begin{equation}
  \label{eq:dPi}
   \d_\ell \Pi_j=
(\d_\ell \Pi_j)\Pi_j +  \Pi_j(\d_\ell \Pi_j).
\end{equation}
We now show: $\forall j\in
\{1,\dots,P\},\ \forall \ell \in \{1,\dots,d\}$, 
\begin{equation}\label{eq:nablaPi}
  \begin{aligned}
  \partial_{\ell}\Pi_j&=\sum_{k\not=j}\( \Pi_k
(\d_\ell \Pi_j)\Pi_j + \Pi_j 
(\d_\ell \Pi_j)\Pi_k \)\\
&= \sum_{1\le k\le P}\( \Pi_k (\d_\ell \Pi_j)\Pi_j + \Pi_j
(\d_\ell \Pi_j)\Pi_k \),
  \end{aligned}
\end{equation}
where the last equality stems from \eqref{eq:plp}. To prove
\eqref{eq:nablaPi}, simply write 
\begin{equation*}
\partial_{\ell}\Pi_j =\sum_{k,m}\Pi_k (\partial_\ell\Pi_j)\Pi_m\,  
\end{equation*}
where we have used $\sum_k\Pi_k={\rm Id}$. Then,  observe that $\Pi_k\Pi_j=\delta_{jk}\Pi_j$ yields
$$\Pi_k (\partial_\ell\Pi_j)+(\partial_\ell\Pi_k) \Pi_j\;\;{\rm whence}\;\; \Pi_k( \partial_\ell\Pi_j)=-(\partial_\ell \Pi_k)\Pi_j.$$
The fact that $\Pi_j\Pi_m=0$ for all $m\not=j$ gives~\eqref{eq:nablaPi}. 

\smallbreak

The last formulas we wish to establish involve the spectral gap. Since
we have a basis of eigenfunctions, we have
\begin{equation*}
  V\Pi_j = \Pi_j V = \l_j \Pi_j. 
\end{equation*}
Differentiating with respect to $x_\ell$, we infer
\begin{equation*}
  (\d_\ell \Pi_j) V + \Pi_j \d_\ell V = \l_j \d_\ell \Pi_j + (\d_\ell
  \l_j)\Pi_j. 
\end{equation*}
For $k\in \{1,\dots,P\}$, multiply this relation by $\Pi_k$ on the
right, and use the property $V\Pi_k=\l_k\Pi_k$:
\begin{equation*}
  \l_k (\d_\ell \Pi_j)\Pi_k + \Pi_j\(\d_\ell V-\d_\ell \l_j\)\Pi_k =
  \l_j (\d_\ell \Pi_j)\Pi_k, 
\end{equation*}
hence 
\begin{equation}\label{eq:dPietV}
  (\l_j-\l_k)(\d_\ell \Pi_j)\Pi_k = \Pi_j\(\d_\ell V-\d_\ell
  \l_j\)\Pi_k. 
\end{equation}
Similarly, we have
\begin{equation}
  \label{eq:7.3'}
  (\l_j-\l_k)\Pi_k (\d_\ell \Pi_j) = \Pi_k\(\d_\ell V-\d_\ell
  \l_j\)\Pi_j. 
\end{equation}

\section{About the growth  of the eigenvectors at infinity}
\label{sec:growthcond}

This section is devoted to the proof of estimates at infinity for the
eigenprojectors associated with a potential $V$ satisfying
Assumption~\ref{assumption}. We will use a lemma on the derivatives
of the inverse of the gap between two different eigenvalues. 
For $j,k\in\{1,\dots, P\}$, $j\not=k$, we recall that we have set
(see \eqref{def:gamma}) 
$$\forall x\in\R,\;\;\gamma_{j,k}(x)=(\lambda_j(x)-\lambda_k(x))^{-1}.$$
Since the results  are not specific to the space dimension one, we prove
them for potentials depending on $x\in\R^d$, $d\ge 1$.

\begin{lemma}\label{lem:gap}
Assume
 \eqref{gapcondition} is satisfied with $n_0\in\N$ and that the
 functions $V$ and $\lambda_j$  ($j\in\{1,\cdots,P\}$) are at most
 quadratic. Then, for $\beta\in\N^d$  and for 
 $j,k\in\{1,\dots, P\}$ with $j\not=k$, 
\begin{eqnarray}
\label{dgap}\left| \partial_x^\beta \gamma_{j,k}(x) \right| & \lesssim
& \<x\>^{n_0+|\beta|(1+n_0)}. 
\end{eqnarray}
\end{lemma}

\begin{proof} 
For $\beta={\bf 1}_\ell$, we immediately obtain
$$\left|\partial_{\ell}\gamma_{j,k}(x)\right|=
\left|{\partial_{\ell}\left(\lambda_j(x)-\lambda_k(x)\right)\over
    \left(\lambda_j(x)-\lambda_k(x)\right)^2}\right|\lesssim
\<x\>^{1+2n_0},$$
from \eqref{gapcondition}, and the fact that $\l_j$ and $\l_k$ are at
most quadratic.

Set $\Lambda_{j,k}=\lambda_j-\lambda_k$: it is at most quadratic. Besides, 
 for $\beta\in\N^d$,  we have 
$$\partial^\beta_x \left(\gamma_{j,k}\right)=
\sum_{{\alpha_1+\cdots+\alpha_p=\beta}\atop{ |\alpha_\ell|\ge  1,\ p\le |\beta|}}
a_{\alpha_1,\dots,\alpha_p}\Lambda_{j,k}^{-1-p} \partial_x^{\alpha_1}
\Lambda_{j,k}\cdots \partial_x^{\alpha_p}\Lambda_{j,k}$$ 
for some real-numbers $a_{\alpha_1,\dots,\alpha_p}$.
The result then follows by observing that 
$$\left|\Lambda_{j,k}^{-1-p} \partial_x^{\alpha_1}
  \Lambda_{j,k}\cdots \partial_x^{\alpha_p}\Lambda_{j,k}\right|\lesssim
\<x\>^{n_0(1+p)}\<x\>^p,$$ 
from \eqref{gapcondition}, and the property $
|\partial_x^{\alpha}\Lambda_{j,k}|\lesssim \<x\>^{(2-|\alpha|)_+}$,
  which follows from the fact that $\Lambda_{j,k}$ is at most
  quadratic, in the sense of Definition~\ref{def:sousquad}. 
\end{proof}

We now consider the eigenprojectors $\Pi_j$ associated with the
eigenvalues $\lambda_j$ of the matrix $V$. Because of the gap
condition, these functions are smooth in $\R^d$. We prove the
following  

\begin{lemma} \label{lem:projectors}
Let $\Pi_j$ be an eigenprojector of $V$ for $j\in\{1,\dots, P\}$, we
have for $\beta\in\N^d$ 
\begin{eqnarray}
\label{dPi} 
| \partial_x^\beta \Pi|_{\C^{N,N}} & \lesssim &
 \<x\>^{|\beta|(1+n_0)},
\end{eqnarray}
where the norm $|\cdot|_{\C^{N,N}}$ denotes the matricial norm.
\end{lemma} 

\begin{proof}
The case $|\beta|=0$ is immediate since $\Pi_j$ is a projector.
In view of~\eqref{eq:nablaPi}, relations~\eqref{eq:7.3'}
and~\eqref{eq:dPietV} imply \eqref{dPi} for $|\beta|=1$.

\smallbreak

We now argue by induction. We suppose that (\ref{dPi}) holds for any
$\gamma\in\N^d$ with $|\gamma|=K$ for some $K\in\N$ and we consider
$\beta$ with $|\beta|=K+1$ and  $\beta_\ell\not=0$.
  Differentiation of order $\beta-{\bf 1}_\ell$ of \eqref{eq:dPi} 
   and multiplication on both sides by $\Pi_j$ yields
    $$\Pi_j(\partial_x^\beta\Pi_j)\Pi_j=\Pi_j\left(\sum_{0<|\alpha|<|\beta|}
      a_\alpha \partial^\alpha_x\Pi_j\partial_x^{\beta-\alpha}\Pi_j\right)\Pi_j,$$ 
   where all along this proof, $a_\alpha$ will denote real numbers
   whose exact value is unimportant. We obtain 
\begin{equation}\label{1}
\left|\Pi_j(\partial_x^\beta\Pi_j)\Pi_j\right|_{\C^{N,N}}\lesssim
\<x\>^{|\beta|(1+n_0)}. 
\end{equation}
Then, for all $k\not=j$, we estimate
  $(\partial_x^\beta \Pi_j) \Pi_k $. To do so, we differentiate
  \eqref{eq:dPietV} and get  
  \begin{align*}
    (\partial^\beta_x \Pi_j )\Pi_k &= \sum _{0<|\alpha|<|\beta|}
   a_\alpha \partial_x^\alpha \Pi_j \partial_x ^{\beta-\alpha}\Pi_k \\
&+ \sum_{\alpha_1+\cdots+\alpha_4=\beta-{\bf 1}_\ell}
b_{\alpha_1,\dots,\alpha_4}\partial_x^{\alpha_1}\left((\lambda_j-\lambda_k)^{-1}
\right)\partial_x^{\alpha_2}\Pi_j \partial_x^{\alpha_3}\partial_{\ell}
\left(V-\lambda_j\right)  
\partial_x^{\alpha_4}\Pi_k.
  \end{align*}
In the first sum, the induction assumption yields
\begin{equation}\label{4}
\left\| \partial_x^\alpha \Pi_j \partial_x^{\beta-\alpha}\Pi_k
\right\|_{\C^{N,N}}\lesssim 
\<x\>^{|\alpha|(1+n_0)+|\beta-\alpha|(1+n_0)}=\<x\>^{|\beta|(1+n_0)}
\end{equation}
 Besides, 
for  each  term in the second sum, we write 
$$\displaylines{\qquad \left\| \partial_x^{\alpha_1}\left((\lambda_j-\lambda_k)^{-1}\right)\partial_x^{\alpha_2}\Pi_j \partial_x^{\alpha_3}\partial_{x_i}\left(V-\lambda_j\right)
\partial_x^{\alpha_4}\Pi_k\right\|_{\C^{N,N}} \hfill\cr\hfill\lesssim \<x\>^{n_0+|\alpha_1|(1+n_0)} \<x\>^{(1-|\alpha_3|)_+}\<x\>^{(1+n_0)(|\alpha_2|+|\alpha_4|)}
\qquad\cr}$$ 
where $r_+=\max(r,0)$ and where we have used the fact that
$V$ and $\lambda_j$ are at most quadratic, together with the induction
assumption and 
Lemma~\ref{lem:gap}. We have the two alternatives: 
\begin{itemize}
\item If $\alpha_3=0$, then 
  \begin{align*}
    n_0+|\alpha_1|(1+n_0)+(1-|\alpha_3|)_+&+(1+n_0)(|\alpha_2|+|\alpha_4|) \\
&= n_0+|\alpha_1|(1+n_0)+1+(1+n_0)(|\alpha_2|+|\alpha_4|)\\
&=(1+n_0)(1+|\alpha_1|+|\alpha_2|+|\alpha_3|)=(1+n_0)|\beta|,
  \end{align*}
since $\alpha_1+\alpha_2+\alpha_4=\beta-{\bf 1}_\ell$.
\item If $\alpha_3\not=0$, then 
  \begin{align*}
    n_0+|\alpha_1|(1+n_0)+(1-|\alpha_3|)_+&+(1+n_0)(|\alpha_2|+|\alpha_4|)\\
& =  n_0+(1+n_0)(|\alpha_1|+|\alpha_2|+|\alpha_4|)\\
&\le (1+n_0)(1+|\alpha_1|+|\alpha_2|+|\alpha_4|)\\
&\le(1+n_0)|\beta|.
  \end{align*}
\end{itemize}
We deduce 
 \begin{equation}\label{5}
 \left\|(\partial_x^\beta\Pi_j)\Pi_k\right\|_{\C^{N,N}}\lesssim 
\<x\>^{|\beta|(1+n_0)},\quad \forall k\not =j.
 \end{equation}
Similarly,
\begin{equation}
  \label{eq:10h22}
   \left\|\Pi_k(\partial_x^\beta\Pi_j)\right\|_{\C^{N,N}}\lesssim 
\<x\>^{|\beta|(1+n_0)},\quad \forall k\not =j.
\end{equation}
In view of \eqref{1}, we infer
\begin{equation}\label{eq:10h26}
  \left\|\Pi_j(\partial_x^\beta\Pi_j)\Pi_k\right\|_{\C^{N,N}} + 
\left\|\Pi_k(\partial_x^\beta\Pi_j)\Pi_j\right\|_{\C^{N,N}}\lesssim 
\<x\>^{|\beta|(1+n_0)},\quad \forall j,k.
\end{equation}
Applying the operator $\d^{\beta-{\bf 1}_\ell}_x$ to
\eqref{eq:nablaPi},  the induction assumption and equations \eqref{5}, \eqref{eq:10h22} and \eqref{eq:10h26}
yield \eqref{dPi}, which concludes the induction. 
\end{proof}

\bibliographystyle{amsplain}
\bibliography{biblio}

\providecommand{\bysame}{\leavevmode\hbox to3em{\hrulefill}\thinspace}
\providecommand{\MR}{\relax\ifhmode\unskip\space\fi MR }
\providecommand{\MRhref}[2]{%
  \href{http://www.ams.org/mathscinet-getitem?mr=#1}{#2}
}
\providecommand{\href}[2]{#2}
\begin{thebibliography}{10}

\bibitem{Bao}
W.~Bao, \emph{Analysis and efficient computation for the dynamics of
  two-component {B}ose-{E}instein condensates}, Stationary and time dependent
  {G}ross-{P}itaevskii equations, Contemp. Math., vol. 473, Amer. Math. Soc.,
  Providence, RI, 2008, pp.~1--26.

\bibitem{BenAbdMehats}
N.~Ben~Abdallah, F.~Castella, and F.~M{\'e}hats, \emph{Time averaging for the
  strongly confined nonlinear {S}chr\"odinger equation, using
  almost-periodicity}, J. Differential Equations \textbf{245} (2008), no.~1,
  154--200.

\bibitem{Ca-p}
R.~Carles, \emph{Nonlinear {S}chr{\"o}dinger equation with time dependent
  potential}, preprint. Archived as \url{http://arxiv.org/abs/0910.4893}, 2009.

\bibitem{CF-p}
R.~Carles and C.~Fermanian, \emph{Nonlinear coherent states and {E}hrenfest
  time for {S}chr\"odinger equations}, Comm. Math. Phys., to appear.

\bibitem{CazCourant}
T.~Cazenave, \emph{Semilinear {S}chr\"odinger equations}, Courant Lecture Notes
  in Mathematics, vol.~10, New York University Courant Institute of
  Mathematical Sciences, New York, 2003.

\bibitem{Fujiwara}
D.~Fujiwara, \emph{Remarks on the convergence of the {F}eynman path integrals},
  Duke Math. J. \textbf{47} (1980), no.~3, 559--600.

\bibitem{GMMP}
P.~G{\'e}rard, P.~A. Markowich, N.~J. Mauser, and F.~Poupaud,
  \emph{Homogenization limits and {W}igner transforms}, Comm. Pure Appl. Math.
  \textbf{50} (1997), no.~4, 323--379.

\bibitem{Hag94}
G.~A. Hagedorn, \emph{Molecular propagation through electron energy level
  crossings}, Mem. Amer. Math. Soc. \textbf{111} (1994), no.~536, vi+130.

\bibitem{HJ98}
G.~A. Hagedorn and A.~Joye, \emph{Landau-{Z}ener transitions through small
  electronic eigenvalue gaps in the {B}orn-{O}ppenheimer approximation}, Ann.
  Inst. H. Poincar\'e Phys. Th\'eor. \textbf{68} (1998), no.~1, 85--134.

\bibitem{Jecko1}
T.~Jecko, \emph{Estimations de la r\'esolvante pour une mol\'ecule diatomique
  dans l'approximation de {B}orn-{O}ppenheimer}, Comm. Math. Phys. \textbf{195}
  (1998), no.~3, 585--612.

\bibitem{Kato}
T.~Kato, \emph{Perturbation theory for linear operators}, Classics in
  Mathematics, Springer-Verlag, Berlin, 1995, Reprint of the 1980 edition.

\bibitem{Lubich}
C.~Lubich, \emph{From quantum to classical molecular dynamics: reduced models
  and numerical analysis}, Zurich Lectures in Advanced Mathematics, European
  Mathematical Society (EMS), Z\"urich, 2008.

\bibitem{SpohnTeufel}
H.~Spohn and S.~Teufel, \emph{Adiabatic decoupling and time-dependent
  {B}orn-{O}ppenheimer theory}, Comm. Math. Phys. \textbf{224} (2001), no.~1,
  113--132, Dedicated to Joel L. Lebowitz.

\bibitem{Teufel}
S.~Teufel, \emph{Adiabatic perturbation theory in quantum dynamics}, Lecture
  Notes in Mathematics, vol. 1821, Springer-Verlag, Berlin, 2003.

\bibitem{TsutsumiL2}
Y.~Tsutsumi, \emph{{$L^2$}--solutions for nonlinear {S}chr\"odinger equations
  and nonlinear groups}, Funkcial. Ekvac. \textbf{30} (1987), no.~1, 115--125.

\end{thebibliography}

\end{document}